\documentclass[11pt,twoside]{article}
\usepackage{bbm}
\usepackage{mathrsfs}
\usepackage{amssymb}
\usepackage{amsmath}
\usepackage{amsthm}
\usepackage{amsfonts}
\usepackage{color}
\usepackage{graphicx}
\usepackage[active]{srcltx}
\usepackage{CJK}
\usepackage{geometry}
\usepackage[cp1252]{inputenc}

\usepackage{mathrsfs}
\usepackage{graphicx}

\usepackage[active]{srcltx}

\allowdisplaybreaks

\usepackage{titletoc}
\titlecontents{section}[0pt]{\addvspace{2pt}\filright}
              {\contentspush{\thecontentslabel\ }}
              {}{\titlerule*[8pt]{.}\contentspage}

\def\rr{{\mathbb R}}

\def\nn{{\mathbb N}}
\def\hh{{\mathbb H}}

\def\ch{\mathcal H}

\def\hn{{{\mathbb H}^n}}

\def\car{{\mathcal R}}

\def\fz{\infty}
\def\az{\alpha}

\def\dz{\delta}

\def\ez{\epsilon}

\def\gz{{\gamma}}

\def\tz{\theta}

\def\wz{\widetilde}

\def\bint{{\ifinner\rlap{\bf\kern.35em--}
\int\else\rlap{\bf\kern.45em--}\int\fi}\ignorespaces}

\def\bbint{{\ifinner\rlap{\bf\kern.35em--}
\hspace{0.078cm}\int\else\rlap{\bf\kern.45em--}\int\fi}\ignorespaces}

\def\lr{\right}
\def\lf{\left}
\def\la{\langle}
\def\ra{\rangle}

\newtheorem{thm}{Theorem}[section]
\newtheorem{lem}[thm]{Lemma}%[section]     %@@!!@@!!
\newtheorem{rem}[thm]{Remark}%[section]    %@@!!@@!!
\newtheorem{defn}[thm]{Definition}%[section]    %@@!!@@!!
\newtheorem{prob}[thm]{Problem}

\numberwithin{equation}{section}

\begin{document}
%\begin{CJK*}{GBK}{song}

\arraycolsep=1pt

\title{\Large\bf
On the Dimension of Kakeya Sets in the First Heisenberg Group
\footnotetext{\hspace{-0.35cm}
\endgraf
 2010 {\it Mathematics Subject Classification:} Primary 28A75; Secondary 28A78 $\cdot$ 28A80
 \endgraf {\it Key words and phrases:}    Kakeya set, Hausdorff dimension, Heisenberg group.
 \endgraf J. L. is supported by the Academy of Finland via the projects: Quantitative rectifiability in Euclidean and non-Euclidean spaces, Grant No. 314172, and Singular integrals, harmonic functions, and boundary regularity in Heisenberg groups,
Grant No. 328846.
}}
\author{Jiayin Liu}
\date{}
\maketitle

\begin{center}
\begin{minipage}{13.5cm}\small
{\noindent{\bf Abstract.}We define Kakeya sets in the Heisenberg group
and show that the Heisenberg Hausdorff dimension of Kakeya sets in the first Heisenberg group is at least 3. This lower bound is sharp since, under our definition, the $\{xoy\}$-plane is a Kakeya set with
Heisenberg Hausdorff dimension 3.
}
\end{minipage}
\end{center}

%\tableofcontents
%\contentsline{section}{\numberline{ } References}{30}

\section{Introduction}

The study of Kakeya sets in Euclidean space is one of the central topics in geometric measure theory.
A set $E \subset\rr^n$ is a Kakeya set if for every $e \in S^{n-1}$ there exists a unit
line segment $I_e$ parallel to $e$ such that $I_e \subset E$.

A natural question is to determine the least Hausdorff dimension of Kakeya sets.

The answer is known in 2-dimensional Euclidean space. Indeed,
Kakeya sets in $\rr^2$ turn out to be of Hausdorff dimension equal to 2 which can be shown by multiple ways. See \cite{d71,f85,m15}.
However, for Kakeya sets in higher dimensional Euclidean space, the sharp lower bound is not known. We would like to remark some progress: Bourgain used two different methods to provide lower bounds \cite{b91,b99}, which were further improved by Wolff \cite{w95} and Katz-Tao \cite{kt99} respectively.
Recently, Katz-Zahl \cite{kz19} and Guth-Zahl \cite{gz18} enhanced the results of \cite{w95} in $\rr^3$ and $\rr^4$ respectively, which show that the best known lower bound is $5/2+\ez_0$ in $\rr^3$ with $\ez_0$ an absolute constant and $3+1/40$ in $\rr^4$. In $\rr^n$ with $n \ge 5$, the best known lower bound $3+(2 -\sqrt{2} )(n - 4)$ was established in \cite{kt02} by Katz-Tao.

As an analogy to Euclidean Kakeya sets, we can define Kakeya sets in the Heisenberg group. In this paper, we denote by $\hn$ the $n$-th Heisenberg group. When $n=1$, we write $\hh$ instead of $\hh^1$ for simplicity.
%We refer the reader to section 2 for other concepts in the following definition.
\begin{defn} \label{d1}
  A set $E \subset\hh^n$ is a Kakeya set if for every unit line segment $I \subset \rr^{2n} \times \{0\}$ centred at the origin, there exists $q \in \hn$ such that $qI \subset E$.
\end{defn}
\noindent
Here and in what follows, by a unit line segment we mean an isometric copy of the unit open interval $(0,1)$. Moreover, by Heisenberg Hausdorff measure we mean the one induced by the Kor\'{a}nyi metric on the first Heisenberg group. For the definition of the Kor\'{a}nyi metric, we refer the readers to Section 2.

The Heisenberg Hausdorff dimension of Euclidean Kakeya sets has been studied in \cite{v14} where the author showed a lower bound on the dimension of Kakeya sets. In \cite{v17}, the author studied Kakeya sets for general metric spaces in axiomatic sense.

According to Definition \ref{d1}, it is not hard to show that the Euclidean Hausdorff dimension of Kakeya sets in $\hh$ is at least 2. This can be done as follows. Under orthogonal projection to the $\{xoy\}$-plane, every Kakeya set $E$ in $\hh$ becomes a Kakeya set $E'$ in $\rr^2$. Hence the known lower bound of the Euclidean Hausdorff dimension of $E'$ is also the one of $E$ since orthogonal projection in $\rr^3$ is Lipschitz. Moreover, 2 is sharp since the $\{xoy\}$-plane is a Kakeya set in $\hh$ with Euclidean Hausdorff dimension 2. However, the Heisenberg Hausdorff dimension of the $\{xoy\}$-plane is 3 and the orthogonal projection from $\hh$ to the $\{xoy\}$-plane is no longer Lipschitz with respect to the Kor\'{a}nyi metric. Hence to calculate a lower bound of Heisenberg Hausdorff dimension of Kakeya sets seems to be a nontrivial problem.

In this note, we will show the following
\begin{thm} \label{thm1}
  In the first Heisenberg group $\hh$ equipped with the Kor\'{a}nyi metric, every Kakeya set has Heisenberg Hausdorff dimension at least 3 and this lower bound is sharp.
\end{thm}

In the following, $E$ will denote a Kakeya set in the first Heisenberg group.

Our method to show Theorem \ref{thm1} is based on the idea of \cite{b64,f85}. We first encode each horizontal line segment in $E$ by a quadruple in $\rr^4$ forming a subset $L(E) \subset \rr^4$. Then we transfer the computation for dimensions of each intersection of $E$ and a plane belonging to a one parameter family to that of a subset in $\rr^3$ obtained by certain projections acting on $L(E)$. This can be seen as a duality principle. Finally we use a recent Marstrand-type projection theorem in $\rr^3$ by K\"{a}enm\"{a}ki-Orponen-Venieri \cite{kov17} and a co-area inequality by Eilenberg-Harrold, Jr. \cite{eh43} to conclude the proof.

\bigskip
Since every Kakeya set in the first Heisenberg group has Heisenberg Hausdorff dimension at least 3, a further question that may be asked is to find a lower bound of 3-dimensional Heisenberg Hausdorff measure among all Kakeya sets in the first Heisenberg group.
Unlike the Euclidean Kakeya set, which may have $n$-dimensional Lebesque measure zero in every $\rr^n$ (for example, see \cite{b19}), it is not easy to show a counterpart for Kakeya sets in Heisenberg group. Hence we would like to ask the following question

\begin{prob}
  Does there exist a Kakeya set in the first Heisenberg group with zero 3-dimensional Heisenberg Hausdorff measure?
\end{prob}

The paper is organised as follows. In section 2, we recall some background in Heisenberg groups, the  Marstrand-type projection theorem in $\rr^3$ and the co-area inequality. In section 3, we prove Theorem \ref{thm1}.

\begin{center}

\textsc{Acknowledgement
}
\end{center}
J. L. would like to thank K. F\"{a}ssler and T. Orponen for many meaningful discussions.

\section{Preliminaries}

The first part of this section is dedicated to a brief introduction to the first Heisenberg group $\mathbb{H}$. For a detailed one, we refer the readers to \cite{cpdt}.

The first Heisenberg group $\mathbb{H}$ is $\mathbb{R}^{3}$,
equipped with the group multiplication, for any $w =  (x ,y , t)$ and $w'=(x',y',t')$, as follows
\begin{equation}\label{mu}
  w  w' = \bigg ( x + x' ,y+y', t+t'+\frac{1}{2}[xy'-x'y] \bigg ).
\end{equation}

We introduce the Kor\'{a}nyi metric on the first Heisenberg group. This is the left invariant
metric given by
\begin{equation}\label{ko}
   d_\hh(w,w'):= \| (w')^{-1}\cdot w\|_\hh
\end{equation}
where $\|\cdot\|_\hh$ is defined as
$$ \|(x,y,t)\|_\hh=   ((x^2+y^2)^2+ 16t^2)^{1/4}. $$

We define horizontal lines in the first Heisenberg group $\hh$ as lines which can be obtained by a left translation of some line passing through the origin and lying in the $\{xoy\}$-plane.

By the definition of horizontal lines, we know that for any $b \in \rr$ and $q \in \hh$, $qI_b$ and $qJ_b$ are horizontal line and horizontal line segment respectively where
  \begin{equation}  \label{pa3}
       I_b(\tau) = (\tau,b\tau,0),   \  \tau \in \rr
  \end{equation}
and
\begin{equation*}
       J_b(\tau) = (\tau,b\tau,0),   \  \tau \in (-\frac{1}{2\sqrt{b^2+1}},\frac{1}{2\sqrt{b^2+1}}).
  \end{equation*}

The following observation is needed in the proof of Theorem \ref{thm1}.

\begin{lem}\label{pa1}
  For any $b \in \rr$ and $q=(q_1,q_2,q_3) \in \hh$,
  \begin{enumerate}
    \item [(1)] $qI_b$ and $qJ_b$ can be parameterised as
  \begin{equation}\label{pa2}
    qI_b(s)= (s,bs+a,-\frac{as}2+d), \ s\in \rr
  \end{equation}
  and
  \begin{equation}\label{pa21}
    qJ_b(s)=  (s,bs+a,-\frac{as}2+d), \ s\in (\ez,\ez+\frac{1}{\sqrt{b^2+1}})
  \end{equation}
  where $a=q_2-bq_1$, $d=q_3+\frac12 a q_1$ and $\ez= q_1-\frac{1}{2\sqrt{b^2+1}}$.
    \item [(2)] If we denote
    \begin{equation}\label{pa31}
      l_{(a,b,d)}:= \{(s,bs+a,-\frac{as}2+d) \in \hh \ | \ s \in \rr \}
    \end{equation}
    and
    $$l^\ez_{(a,b,d)}:= \lf \{(s,bs+a,-\frac{as}2+d) \in \hh \ | \ s \in (\ez,\ez+\frac{1}{\sqrt{b^2+1}}) \lr \}.$$
    Then
    $l_{(a,b,d)}^\ez$ has length $1$ with respect to $d_\hh$ for every $a,b,d,\ez$.
    \item [(3)]
    If $b \in (-\sqrt{3},\sqrt{3})$, then the orthogonal projection of $l_{(a,b,d)}^\ez$ to the $x$-axis has Euclidean length greater than $\frac12$.
  \end{enumerate}
\end{lem}

\begin{proof}
\begin{enumerate}
  \item [(1)] Let $I_b$ be parameterised as in \eqref{pa3}. Then by the Heisenberg multiplication law \eqref{mu}, we have
  $$ qI_b = \{(q_1+\tau, q_2 +b\tau, q_3+ \frac12(q_1b\tau-q_2 \tau)) \ | \ \tau \in \rr \ \} . $$
  Letting $s=q_1+\tau$, $a=q_2-bq_1$ and $d=q_3+\frac12 a q_1$, we arrive at \eqref{pa2}. In addition, letting $\ez= q_1-\frac{1}{2\sqrt{b^2+1}}$, we verify that \eqref{pa21} holds.
  \item [(2)] Using the definition of $d_\hh$ and the fact that left translation is an isometry with respect to $d_\hh$, we deduce the result.
  \item [(3)] From \eqref{pa21}, the orthogonal projection of $l_{(a,b,d)}^\ez$ to the $x$-axis is the interval $(\ez,\ez+\frac{1}{\sqrt{b^2+1}}) \subset x$-axis. Hence when $b \in (-\sqrt{3},\sqrt{3})$, the length of the interval is greater than $\frac12$.
\end{enumerate}
\vspace{-0.5cm}
\end{proof}

\begin{rem}
  \rm In the sense of sub-Riemannian geometry, there exists more general horizontal curves in $\hh$ besides horizontal lines. Indeed, associated to the group operation \eqref{mu}, we can define the left invariant vector fields
    \begin{align*}
        X &= \frac{\partial}{\partial x} - \frac{y}{2}\frac{\partial}{\partial t},  \quad Y = \frac{\partial}{\partial y} + \frac{x}{2}\frac{\partial}{\partial t}.
    \end{align*}
A Lipschitz curve $\gz = (\gz_1, \gz_{2}, \gz_{3}) : [a, b] \to \hh$ is said to be horizontal if $\dot{\gz}(s) \in \text{Span}\{X(\gz(s)),Y(\gz(s))\}$, i.e. $ \dot{\gz}(s) = a(s)X(\gz(s))+b(s)Y(\gz(s))$, for almost every $s \in [a, b]$.
One can check that horizontal lines are indeed horizontal curves under the this definition. For more information from the sub-Riemannian point of view, we refer readers to \cite[Chapter 2]{cpdt}.
\end{rem}

In this paper,
we denote by $\ch_{\hh}^s$ (resp. $\ch_{\rr}^{s}$) the $s$-dimensional Hausdorff measure induced by the Kor\'{a}nyi metric (resp. Euclidean metric) and by $ \dim_{\ch}^{\hh}$ (resp. $ \dim_{\ch}^{\rr}$) the Hausdorff dimension of sets induced by Kor\'{a}nyi metric (resp. Euclidean metric).
In addition, given any set $A \subset \hh$, we denote
\begin{equation}\label{pa6}
   L(A):= \{ (a,b,d,\ez)\in\rr\times (-\sqrt{3},\sqrt3)\times\rr \times \rr \ | \  l^\ez_{(a,b,d)} \subset A \}
\end{equation}
and
\begin{equation}\label{pa7}
   L(A,c):= \{ (a,b,d,\ez)\in L(A) \ | \  l^\ez_{(a,b,d)} \cap \{x=c\} \ne \emptyset \}.
\end{equation}

\bigskip
Next, we recall a version of a Marstrand-type projection theorem \cite[Theorem 1.2]{kov17}:
\begin{thm} \label{thm2}
  Suppose that $\gz : [0, 2\pi) \to \rr^3, \tz \mapsto \gz(\tz)=\frac1{\sqrt{2}}(\cos \tz, \sin \tz ,1)$. If $K \subset \rr^3$ is Borel
set, then $\dim_{\ch}^{\rr} \rho_{\gz(\tz)}(K) = \min \{ \dim_{\ch}^{\rr}K, 1\}$ for almost every $\tz \in [0, 2\pi)$.
\end{thm}
Here, and in what follows, for any $x \in \rr^3\setminus \{0\}$, $\rho_{x}: \rr^3 \to \text{Span} (x)$  denotes the Euclidean orthogonal projection to the straight line passing through the origin and $x$.

 We also need the following co-area inequality \cite[Theorem 1]{eh43}:
 \begin{thm} \label{thm3}
   Let $X$ be an arbitrary metric space, $0 \le \az  < \fz$ be real numbers
and $F \subset X$ be any subset. Then, for any $1$-Lipschitz map $f : X \to \rr$ we have
\begin{equation}\label{co}
  \int_\rr^\ast \ch^{\az}_X(F \cap f^{-1}(y)) \, d y \le \ch^{\az+1}_X(F) .
\end{equation}
Here, $\int^\ast_\rr g \, dy$ is the upper integral of $g: \rr \to [0,+\fz)$. That is
$$ \int^\ast_\rr g(y) \, dy = \inf \int_\rr h(y) \, dy $$
where the infimum is taken over all measurable functions $h: \rr \to [0,+\fz) $ satisfying $0 \le g(y) \le  h(y)$
for a.e. $y \in \rr$.
 \end{thm}

\section{Proof of Theorem 1.2}

\begin{proof}[Proof of Theorem \ref{thm1}]
Since every set is contained in a $G_\dz$-set of the same dimension, we may assume $E$ to be $G_\dz$.
\bigskip

\emph{Step 1. Properties of $L(E,c)$} \\
First, we need

\bigskip
\noindent
\textsc{Claim I:} \emph{For every $c \in \rr$, $L(E,c)$ is a $G_\dz$ set in $\rr^4$.}

\begin{proof}[Proof of \textsc{Claim I}]
Recalling \eqref{pa6},
$$  L(E,c) = \{ (a,b,d,\ez)\in L(E) \ | \ l^\ez_{(a,b,d)} \cap \{x=c\}\ne \emptyset \}.$$
%If $L(E,c)=\emptyset$, then the conclusion trivially holds. Below we assume $L(E,c) \ne \emptyset$.
%In \textsc{Claim II} we will show that there exists $c\in \rr$ such that $L(E_i,c)$ is not empty.
Since $E$ is a $G_\dz$ set, we can find a sequence of open sets $\{E_i\}_{i \in \nn}$ such that
$ E_i \supset E_{i+1}$ for each $i \in \nn$ and
$$  E = \bigcap_{i \in \nn} E_i.$$

 Consider the sets
 $$ L(E_i,c) = \{  (a,b,d,\ez)\in L(E_i) \ | \ (l^\ez_{(a,b,d)} \cap \{x = c\}) \ne \emptyset  \}.$$
 We assert $L(E_i,c)$ is open for any $c \in \rr$ and $i \in \nn$.
 %if
% \begin{equation}\label{emp}
%   l_{(a,b,d)} \cap S(c,i)= \emptyset,
% \end{equation}
%  then by $\emptyset \subset E_i \cap S(c,i)$ we know
% $$(a,b,d) \in L(E_i,c).$$
% Note that $l_{(a,b,d)}$ is compact and $S(c,i)$ is open. By \eqref{emp} we infer that
% there exists a neighbourhood $\cu(l_{(a,b,d)})$ of $l_{(a,b,d)}$ in $\hh \setminus S(c,i)$ which further implies for $(a',b',d')$ close enough to $(a,b,d)$, $l_{(a',b',d')} \subset \cu(l_{(a,b,d)})$.
% Hence we have
% $$  l_{(a',b',d')} \cap S(c,i) = \emptyset  \subset  E_i \cap S(c,i)$$
% and
% $ (a',b',d') \in L(E_i,c)$.
 Consider an arbitrary quadruple $(a,b,d,\ez) \in L(E_i,c)$, i.e.
 \begin{equation}\label{nemp}
   l^\ez_{(a,b,d)} \cap \{x=c\} \ne \emptyset \quad \text{and} \quad  l^\ez_{(a,b,d)}   \subset  E_i,
 \end{equation}
 Thanks to the openness of $E_i$ and the interval $(\ez,\ez+\frac{1}{\sqrt{b^2+1}})$, we deduce that for $(a',b',d',\ez')$ close enough to $(a,b,d,\ez)$,
 we have
 $$  l^{\ez'}_{(a',b',d')} \cap \{x=c\} \ne \emptyset \quad \text{and} \quad l^{\ez'}_{(a',b',d')}   \subset  E_i$$
 and hence
 $ (a',b',d',\ez') \in L(E_i,c)$,
 which implies $L(E_i,c)$ is open.

 We are left to show
  \begin{equation}\label{open}
    L(E,c)= \bigcap_{i \ge 1} L(E_i,c).
  \end{equation}
Note that the direction ``$\subset$'' of \eqref{open} is obvious. We show the direction ``$\supset$''. If there exists $(a,b,d,\ez) \in \bigcap_{i \ge 1} L(E_i,c)$, then we infer that
\begin{equation}\label{nemp1}
 l^\ez_{(a,b,d)} \cap \{x=c\} \ne \emptyset \quad \text{and} \quad  l^\ez_{(a,b,d)}   \subset  E_i, \text{ for all } i.
\end{equation}
On the other hand, $E = \cap_{i \in \nn} E_i$ implies
\begin{equation}\label{nemp2}
  \{x=c\} \cap \lf(\bigcap_{i \in \nn} E_i \lr ) = \{x=c\} \cap E.
\end{equation}
Combining \eqref{nemp1} and \eqref{nemp2} we conclude
$$ l^\ez_{(a,b,d)}   \subset E ,$$
which verifies \eqref{open}.
\end{proof}

Furthermore,
we have the following

\bigskip
\noindent
\textsc{Claim II:} \emph{There exists at least one $c_0 \in \rr$ satisfying
\begin{equation}\label{cont}
   \ch^1_\rr(\pi_{123}(L(E,c_0))) >0
\end{equation}
where $\pi_{123}$ is the orthogonal projection from $\rr^4$ to the subspace spanned by the first three coordinates.}

\begin{proof}[Proof of \textsc{Claim II}]
    Since $E$ is a Kakeya set in $\hh$,
by Definition \ref{d1} and recalling \eqref{pa6}, we have
   $$ \pi_2 (L(E)) \supset (-\sqrt {3},\sqrt {3}) $$
   where $\pi_2$ is the orthogonal projection from $\rr^4$ to the subspace spanned by the second coordinate,
   which implies
   \begin{equation}\label{pr6}
     \ch^1_\rr(\pi_{2}(L(E))) >0.
   \end{equation}
   By observing that
   $$ L(E)= \bigcup_{c \in \mathbb{Q}} L(E,c) $$
   and using \eqref{pr6},
   we infer that there exists $c_0$ such that
   \begin{equation}\label{pr7}
     \ch^1_\rr(\pi_{2}(L(E,c_0)))>0.
   \end{equation}
   Noting that
   $$ \ch^1_\rr(\pi_{123}(L(E,c))) \ge \ch^1_\rr(\pi_{2}(L(E,c))) , \ \forall c \in \rr,$$
   we conclude the proof.
\end{proof}

We end step 1 with the following

\bigskip
\noindent
\textsc{Claim III:}
\emph{
We can find a Borel set $B \subset \pi_{123}(L(E,c_0))$ with
\begin{equation}\label{wz0}
  \ch_{\rr}^{1}(B) >0
\end{equation} and
at least one of the following holds:
\begin{equation}\label{wz01}
  B \subset \pi_{123}(L(E,c)),  \quad  \forall c \in [c_0-\frac14,c_0]
\end{equation}
or
\begin{equation}\label{wz1}
  B \subset \pi_{123}(L(E,c)),  \quad  \forall c \in [c_0,c_0+\frac14].
\end{equation}
  }

\begin{proof}[Proof of \textsc{Claim III}] Using Lemma \ref{pa1}(3), we observe that if $l^\ez_{(a,b,d)} \cap \{x=c_0\} \ne \emptyset$, then either
$$   l^\ez_{(a,b,d)} \cap \{x=c_0-\frac14\} \ne \emptyset, $$
or
$$   l^\ez_{(a,b,d)} \cap \{x=c_0+\frac14\} \ne \emptyset. $$

\vspace*{7pt}
\begin{figure}[h]
\centering
\includegraphics[width=12cm]{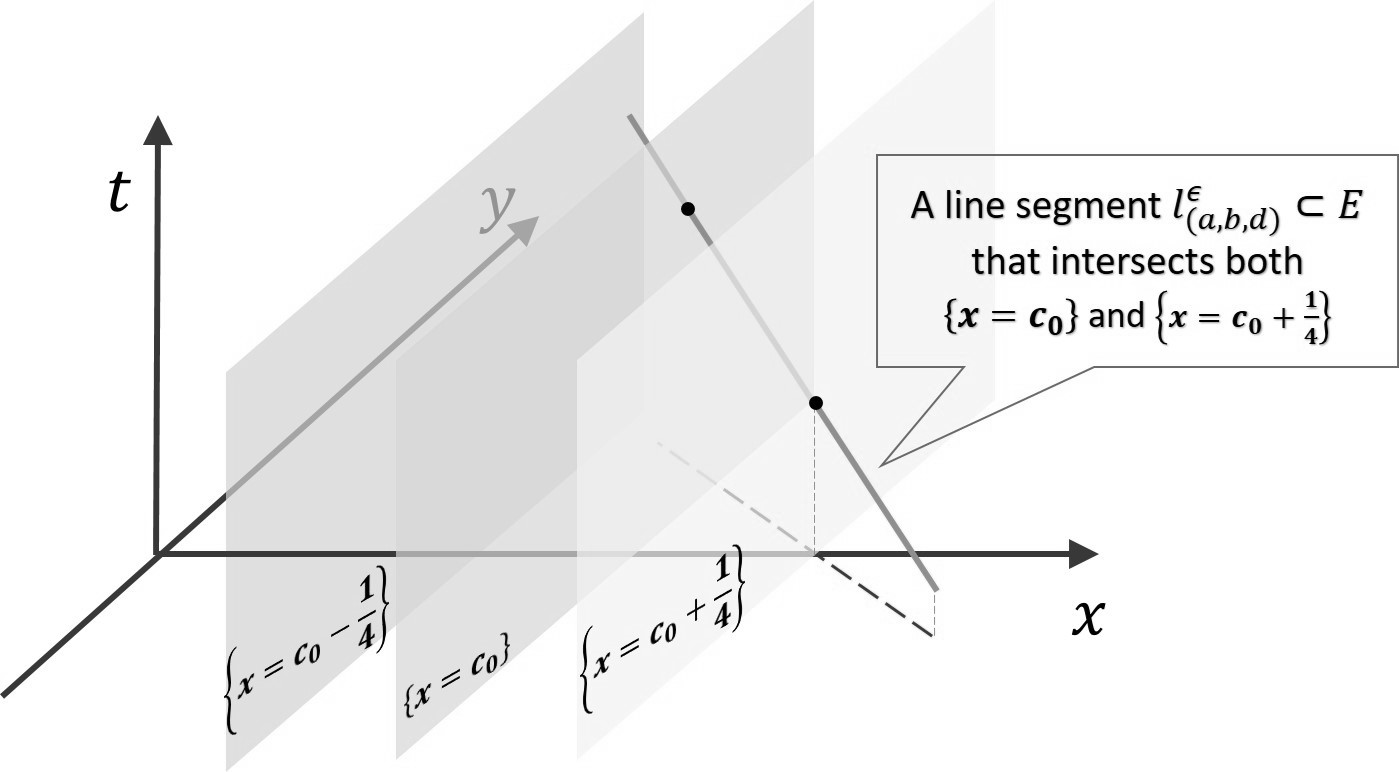}
\caption  {A line segment that intersects both $\{x=c_0\}$ and $\{ x= c_0 + \frac14\}$.}
\label{f1}
\end{figure}

\noindent
Hence we have
$$ L (E,c_0) \subset [L (E,c_0) \cap L (E,c_0-\frac14)] \cup [L (E,c_0) \cap L (E,c_0+\frac14)].$$
We conclude
\begin{equation}\label{alter}
  \pi_{123} (L (E,c_0)) \subset \pi_{123} (L (E,c_0) \cap L (E,c_0-\frac14)) \cup \pi_{123} (L (E,c_0) \cap L (E,c_0+\frac14)).
\end{equation}
From the above inclusion, we can assume, without loss of generality, that
$$ \ch^1_\rr \lf ( \pi_{123} (L (E,c_0) \cap L (E,c_0+\frac14)) \lr) \ge \frac12 \ch^1_\rr(L (E,c_0)) >0 $$
where the last inequality results from \textsc{Claim II}.

On the other hand, if $(a,b,d,\ez) \in L (E,c_0) \cap L (E,c_0+\frac14)$, then for any $c \in [c_0,c_0+\frac14]$, we have $(a,b,d,\ez) \in L (E,c)$, which indicates
$$  L (E,c_0) \cap L (E,c_0+\frac14) \subset L (E,c) \text{  for any } c \in [c_0,c_0+\frac14]. $$
By \textsc{Claim I}, for any $c \in \rr$, we know that $L (E,c)$ is a $G_\dz$ set and hence $\pi_{123} (L (E,c_0) \cap L (E,c_0+\frac14))$ is an analytic set.
Hence we can apply Corollary 2 in \cite{d52} to choose $B$ to be a closed subset of $\pi_{123} (L (E,c_0) \cap L (E,c_0+\frac14))$ with $\ch^1_\rr(B)>0$. Therefore $B$ satisfies the assumption of the claim.
\end{proof}

\textsc{Claim III} enables us to choose $c_0 \in\rr$ such that, without loss of generality, there exists a Borel set $B \subset \pi_{123}(L(E,c_0))$ satisfying \eqref{wz0}, i.e.
$$\ch_{\rr}^{1}(B) >0.$$
and \eqref{wz1}.
Therefore,
we infer that
\begin{equation}\label{dd2}
  \dim_{\ch}^{\rr}(B) \ge 1.
\end{equation}

\bigskip
\emph{Step 2. Establish a duality principle.}  \\
Recall the definition of $l_{(a,b,d)}$ in \eqref{pa31}. For every $c \in [c_0,c_0+1/4]$, we consider $E_c \subset \{x=c\} \cap E \subset \hh$ defined by
\begin{align}\label{lc}
  E_c :&= \{l_{a,b,d} \cap \{x=c\} \ | \ (a,b,d) \in B\} \notag \\
   & =\{(c,bc+a,-\frac{ac}2+d) \ | \ (a,b,d) \in B\}.
\end{align}

Use left translation $T_{(-c,0,0)}$ to translate $E_c$ to $\{yot\}$-plane,
which means $E_c^1:=T_{(-c,0,0)}(E_c)$ lies in $\{yot\}$-plane and has same dimension as $E_c$.
See the above Figure \ref{f2}.

%By Claim$(\ast\ast)$, for any $c \in [c_0,c_0+1/4]$, we have
%\begin{equation}\label{dd3}
%  \dim_{\ch}^{\rr}(L(E,c))=\dim_{\ch}^{\rr}(B) \ge 1.
%\end{equation}

\vspace*{7pt}
\begin{figure}[h]
\centering
\includegraphics[width=10cm]{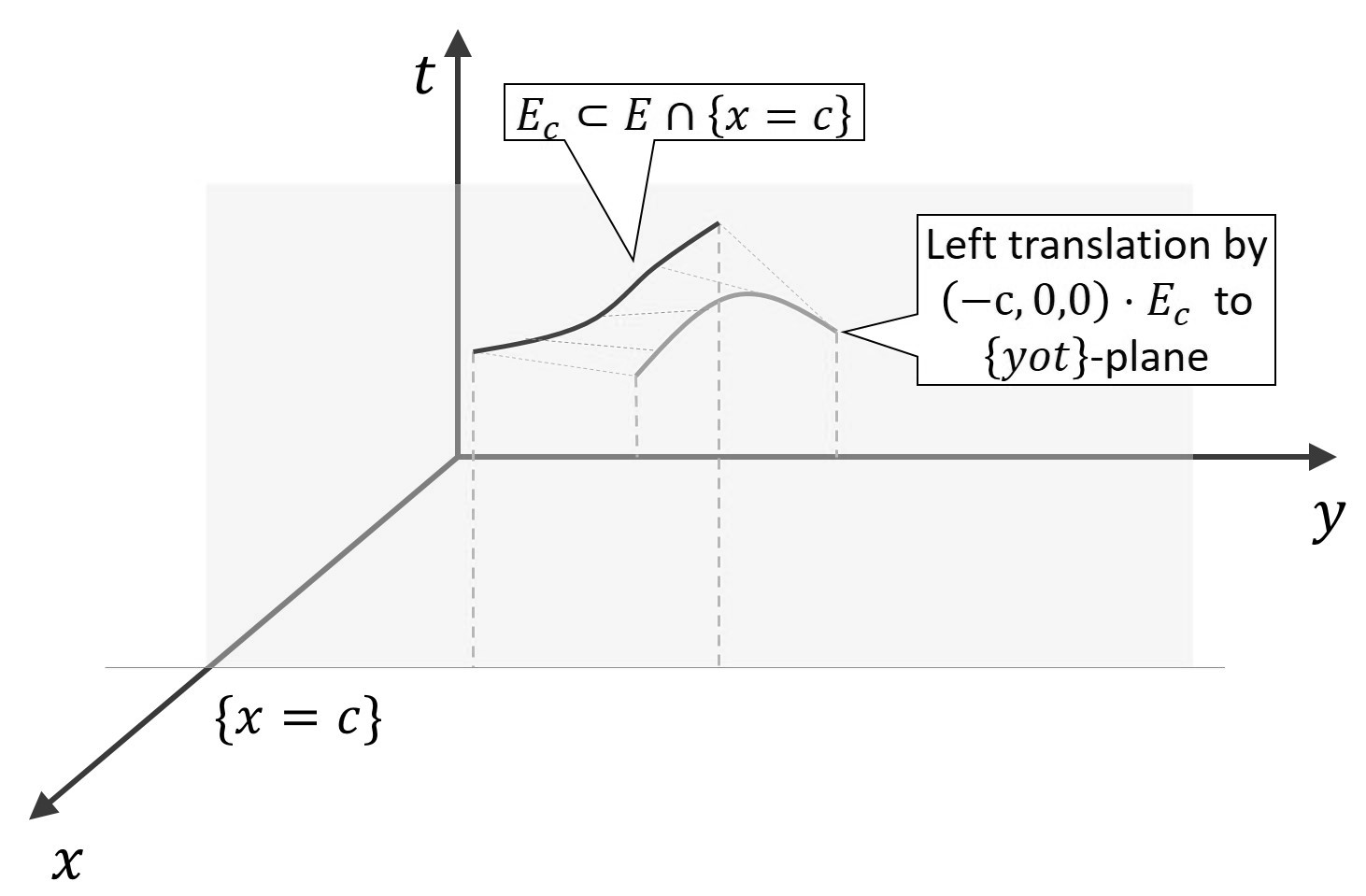}
\caption  {Translating $E_c$ to $\{yot\}$-plane}
\label{f2}
\end{figure}

Recalling \eqref{lc} and the Heisenberg multiplication law \eqref{mu}, we deduce that
\begin{align}\label{tr}
  E_c^1 &= T_{(-c,0,0)}(E_c) \notag \\
   & = \{(-c,0,0)\cdot(c,bc+a,-\frac{ac}2+d) \ | \ (a,b,d) \in B \} \notag \\
   &= \{  (0,bc+a,-ac-\frac{bc^2}2+d) \ | \ (a,b,d) \in B     \} .
\end{align}
Notice that the third coordinate of points in $E_c^1$ expressed in \eqref{tr} takes the form
\begin{equation*}
  -ac-\frac{bc^2}2+d = \lf \la (-c,-\frac{c^2}2,1),(a,b,d) \lr \ra
\end{equation*}
where $\la,\ra$ is the Euclidean inner product in $\rr^3$.

By considering the $t$-axis as $\rr$ and letting
\begin{equation}\label{phi}
  \varphi: \{yot\} \to \hh, \ (0,y,t) \mapsto (0,0,t),
\end{equation}
we can write
\begin{align}\label{phi2}
  \varphi(E_c^1) &= \lf\{\lf \la (-c,-\frac{c^2}2,1),(a,b,d) \lr \ra \ | \ (a,b,d) \in B \lr\} \notag\\
  & = \lf( 1+ \frac{c^2}{2} \lr)\rho_{(-c, -\frac{c^2}{2},1)}(B).
\end{align}

Equation \eqref{phi2} implies that $\varphi(E_c^1)$ can be viewed as a Euclidean projection of $B$ to the one parameter family of lines $\Gamma=\{\gz_c: \rr \to \rr^3 \ | \ t \mapsto(-ct, -\frac{c^2}{2}t,t), c \in [c_0,c_0+1/4]\}$ up to scalings. Letting $t=1$, we observe that $c \mapsto (-c,-\frac{c^2}{2},1)$ forms a part of parabola $\chi$ in $\rr^3$. Hence this one parameter family of lines forms part of a cone $C_1$ in $\rr^3$, i.e.
$$ C_1=\{ (x,y,z) \in \rr^3\ | \ x^2=-2yz \}. $$
Moreover, the intersection of $\Gamma$ and the unit sphere in $\rr^3$ is contained in a circle and can be parameterised as
$$\wz \gz(c) = \bigg \{ \frac{2}{2+c^2}(-c,-\frac{c^2}{2},1) \ | \ c\in [c_0,c_0+1/4] \bigg \}.$$
We see the arc $\wz \gz$ and the parabola $\chi$ are both conical curves, they can be included in one same cone as Figure \ref{f3} depicts.

\vspace*{7pt}
\begin{figure}[h]
\centering
\includegraphics[width=12cm]{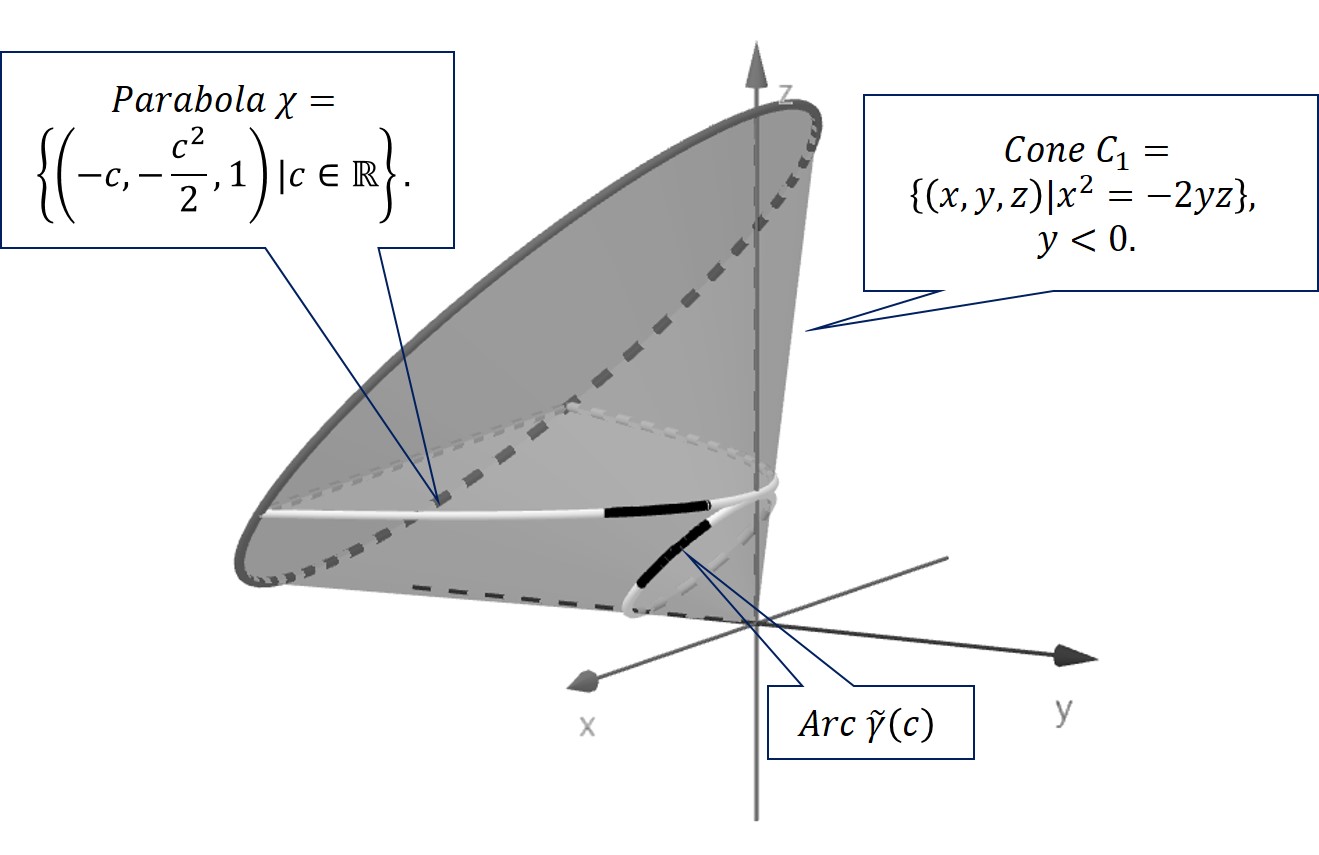}
\caption  {Parabola $\chi$ and Arc $\wz \gz$ can be included in one same cone}
\label{f3}
\end{figure}

\bigskip
\emph{Step 3. Conclusion.}  \\
In Theorem \ref{thm2}, the family of lines passing through the origin and  $\gz(\tz)=\frac1{\sqrt{2}}(\cos \tz, \sin \tz ,1)$ also spans a cone $C_2=\{ (x,y,z) \in \rr^3\ | \ x^2+y^2=z^2\}$ in $\rr^3$.

We observe that the cone $C_2$ can be obtained by a rotation $\car: (x,y,z)\mapsto(x,\frac{\sqrt{2}}{2}(y+z),\frac{\sqrt{2}}{2}(z-y))$ acting on $C_1$ and $\wz \gz$ is mapped to an arc of $\gz$, i.e.
$$ \gz(\tz(c)) \!= \!\car \circ \wz \gz(c)\! = \!\frac{2}{2+c^2}\lf (\! -c,\frac{\sqrt{2}}{4}(2-c^2),\frac{\sqrt{2}}{4}(2+c^2) \! \lr ) \! = \! \frac{1}{\sqrt{2}} \lf ( \frac{-2\sqrt{2}c}{2+c^2}, \frac{2-c^2}{2+c^2},1 \lr )\hspace{-1mm}, \, c \in \! [c_0,c_0+\frac14],$$
where $\tz(c)$ is determined from
$$ (\cos(\tz(c)),\sin(\tz(c)))= \lf ( \frac{-2\sqrt{2}c}{2+c^2}, \frac{2-c^2}{2+c^2} \lr ). $$
 This implies
 \begin{equation}\label{ac}
   \rho_{(-c, -\frac{c^2}{2},1)}((x,y,z))= \rho_{\gz(\tz(c))}(\car(x,y,z)), \quad \forall (x,y,z)\in \rr^3.
 \end{equation}

%Recalling that $E$ is assumed to be $G_\dz$, we have
%$$  L(E)=\{ (a,b,d) \in \rr^3 \ | \ l_{a,b,d} \subset E\} =  \bigcap_{r=1}^{\fz} \{ (a,b,d) \in \rr^3 \ | \ l_{a,b,d} \cap B_{\hh}(0,r) \subset E \cap B_{\hh}(0,r)\}$$
%is also $G_\dz$ and hence Borel measurable in $\rr^3$ where $B_{\hh}(0,r)$ is the ball centered at origin with radius $r$ under Kor\'{a}nyi metric.

By \textsc{Claim III} and \eqref{dd2}, we know $B$ is Borel and $\dim_{\ch}^{\rr}(B) \ge 1$.  We use \eqref{ac} and apply Theorem \ref{thm2} to the family of lines passing through the origin and $\{\gz(\tz(c))\}_{c\in[c_0,c_0+1/4]}$ to deduce that
\begin{equation}\label{d2}
  \dim_{\ch}^{\rr}[\rho_{(-c, -\frac{c^2}{2},1)}(B )] = \dim_{\ch}^{\rr}[\rho_{\gz(\tz(c))}(\car(B) )]  =1 \quad a.e. \ c\in [c_0,c_0+1/4].
\end{equation}
Recalling the definition of $\varphi$ in \eqref{phi} and according to \eqref{ko}, we know $\varphi$ is $1$-Lipschitz with respect to $d_\hh$ and for any set $A \subset t$-axis, $$\dim_\ch^\hh(A)=2\dim_\ch^\rr(A).$$
Hence for any $c \in [c_0,c_0+1/4]$ such that \eqref{d2} holds, combining \eqref{lc}, \eqref{phi2}, \eqref{d2} and the above equality, we conclude
\begin{align*}
  \dim_{\ch}^{\hh}(\{x=c\} \cap E) & \ge \dim_{\ch}^{\hh}(E_c)= \dim_{\ch}^{\hh}(E_c^1) \ge \dim_{\ch}^{\hh}(\phi(E_c^1))\\
   &=2\dim_{\ch}^{\rr}(\rho_{(-1, \frac{c}{2},c^2)}(B ))\\
   &=2
\end{align*}
and for any $0<\az <2$, we deduce
$$  \ch_\hh^\az (\{x=c\} \cap E) =\fz.$$
By definition of $d_\hh$, the map $f:(\hh,d_\hh) \to \rr, (x,y,t) \to (x,0,0)$ is $1$-Lipschitz. Now letting $X=\hh$, $Y=[c_0,c_0+1/4]$ and $F=E\cap\{(x,y,t)\in \hh \ | \ x \in [c_0,c_0+1/4] \}$ in Theorem \ref{thm3}, for any $0<\az <2$, we derive
$$  \ch_\hh^{\az+1} (F) \ge \int^\ast_{[c_0,c_0+1/4]} \ch_\hh^\az (F \cap f^{-1}(y)) \, dy =\fz ,$$
which implies
$$ \dim_{\ch}^{\hh}(E)\ge \dim_{\ch}^{\hh}(E\cap\{(x,y,t)\in \hh \ | \ x \in [c_0,c_0+1/4] \}) =\dim_{\ch}^{\hh}(F) \ge 3.$$
We finish the proof.
\end{proof}

\noindent Jiayin Liu

\noindent
Department of Mathematics and Statistics, University of Jyv\"{a}skyl\"{a}.

\noindent{\it E-mail }:  \texttt{jiayin.mat.liu@jyu.fi}

\end{document}